\documentclass[11pt]{amsart}
\usepackage{amssymb}

\newcommand{\non}{\mathrm{non}}
\newcommand{\add}{\mathrm{add}}

\newcommand{\cof}{\mathrm{cof}}
\newcommand{\M}{\mathcal{M}}
\newcommand{\K}{\mathcal{K}}
\newcommand{\F}{\mathcal{F}}
\newcommand{\C}{\mathcal{C}}

\newcommand{\A}{\mathcal A}
\newcommand{\w}{\omega}
\newcommand{\IR}{\mathbb R}
\newcommand{\IQ}{\mathbb Q}

\newcommand{\II}{\mathbb I}
\newcommand{\e}{\varepsilon}

\newcommand{\BB}{\mathbf{B}}

\newcommand{\W}{\mathcal W}
\newcommand{\U}{\mathcal U}

\newcommand{\V}{\mathcal V}
\newcommand{\cl}{\operatorname{cl}}
\newcommand{\Int}{\operatorname{Int}}
\newcommand{\Ra}{\Rightarrow}
\newcommand{\diam}{\mathrm{diam}}

\newtheorem{theorem}{Theorem}[section]
\newtheorem{problem}[theorem]{Problem}
\newtheorem{proposition}[theorem]{Proposition}
\newtheorem{corollary}[theorem]{Corollary}
\newtheorem{lemma}[theorem]{Lemma}

\theoremstyle{definition}

\newtheorem{remark}[theorem]{Remark}

\title{Hereditarily supercompact spaces}
\author{Taras Banakh, Zdzis\l aw Koszto\l owicz, S\l awomir Turek}
\address{T.Banakh: Ivan Franko National University of Lviv (Ukraine) and Jan Kochanowski University in Kielce (Poland)}
\email{t.o.banakh@gmail.com}
\address{Z.Koszto\l owicz, S.Turek: Jan Kochanowski University in Kielce (Poland)}
\email{zdzisko@ujk.kielce.pl, sturek@ujk.kielce.pl}
\subjclass{}
\keywords{Hereditarily supercompact space, monotonically normal space}

\begin{document}

\begin{abstract} A topological space $X$ is called {\em hereditarily supercompact} if each closed subspace of $X$ is supercompact.
By a combined result of Bula, Nikiel, Tuncali, Tymchatyn, and Rudin, each monotonically normal compact Hausdorff space is hereditarily supercompact. A dyadic compact space is hereditarily supercompact if and only if it is metrizable. Under (MA$+\neg$CH) each separable hereditarily supercompact space is hereditarily separable and hereditarily Lindel\"of. This implies that under (MA$+\neg$CH) a scattered compact space is metrizable if and only if it is separable and hereditarily supercompact. The hereditary supercompactness is not productive: the product $[0,1]\times \alpha D$ of the closed interval and the one-point compactification $\alpha D$ of a discrete space $D$ of cardinality $|D|\ge \non(\M)$ is not hereditarily supercompact (but is Rosenthal compact and uniform Eberlein compact).
Moreover, under the assumption $\cof(\M)=\w_1$ the space $[0,1]\times\alpha D$ contains a closed subspace $X$ which is first countable and hereditarily paracompact but not supercompact.
\end{abstract}

\maketitle

\section{Introduction}

This paper is devoted to studying hereditarily supercompact spaces. By definition, a topological space $X$ is {\em hereditarily supercompact} if each closed subspace $Z$ of $X$ is supercompact.

We recall that a topological space $X$ is called {\em supercompact} if it has  a subbase of the topology such that each cover of $X$ by elements of this subbase has a two element subcover. By the Alexander Lemma (see e.~g.~\cite[3.12.2~(a)]{En}), each
supercompact space is compact. The theory of supercompact spaces was intensively developed in 70-90-ies of XX century, see \cite{Bell}, \cite{BNTT}, \cite{Dem}, \cite{JvM}, \cite{vMM}, \cite{Mills}, \cite{SS}.

According to a result of Strok and Szymanski \cite{SS}, each compact metrizable space is supercompact. Alternative proofs of this important results were given by van Douwen \cite{Douwen}, Mills \cite{Mills} and D\c ebski \cite{Dem}. The most general result in this direction was obtained by Bula, Nikiel, Tuncali, Tymchatyn \cite{BNTT} and Rudin \cite{Rudin} who proved that each monotonically normal compact space is supercompact and hence hereditarily supercompact. On the other hand, a dyadic compact space (in particular, a Tychonoff or Cantor cube) is hereditarily supercompact if and only if it is metrizable, see Theorem~\ref{dyadic}. In Corollary~\ref{c:MACH} we prove that under (MA$+\neg$CH), each separable hereditarily supercompact space is hereditarily separable and hereditarily Lindel\"of.
This implies that under (MA$+\neg$CH) a scattered compact space is metrizable if and only if separable and hereditarily supercompact.

In this paper we present several constructions of hereditarily supercompact spaces and shall prove that the hereditary supercompactness is not productive: the closed interval $[0,1]$ and the one-point compactification $\alpha D$ of a discrete space $D$ of cardinality $|D|\ge\non(\M)\ge\aleph_1$ are hereditarily supercompact but their product $[0,1]\times\alpha D$ is not. Here $\non(\M)$ denotes the smallest cardinality of a non-meager subset of the real line. Under certain Set-Theoretic Assumptions (namely, $\cof(\M)=\w_1$), the product $[0,1]\times\alpha D$ contains a closed subspace $X$ which is first countable, hereditarily paracompact, but not supercompact. Being a closed subspace of $[0,1]\times\alpha D$, the space $X$ is uniform Eberlein and Rosenthal compact.

\section{Some properties of hereditarily supercompact spaces}

In this section we study on the interplay between the class of hereditary supercompact spaces and other classes of compact spaces and present several constructions preserving the hereditary supercompactness. All topological spaces considered in this paper are regular and hence Hausdorff. For a subset $A$ of a topological space $X$ by $\cl(A)$ and $\Int(A)$ we shall denote the closure and interior of $A$ in $X$, respectively. By $[X]^{<\w}$ we shall denote the family of all non-empty finite subsets of a set $X$.

To detect supercompact spaces we shall apply a characterization of the supercompactness in terms of binary closed $k$-networks.

We recall that a family $\K$ of subsets of a topological space $X$ is called
\begin{itemize}
\item {\em linked} if $A\cap B\ne\emptyset$ for any subsets $A,B\in\K$;
\item {\em binary} if each finite linked subfamily $\mathcal L\subset\K$ has non-empty intersection $\bigcap\mathcal L\ne\emptyset$;
\item a {\em $k$-network} if for each open set $U\subseteq X$ and a compact subset $K\subseteq U$ there is a finite subfamily $\F\subseteq\K$ such that  $K\subseteq\bigcup\F\subseteq U$;
\item a {\em closed $k$-network} if $\K$ is a $k$-network consisting of closed subsets of $X$.
\end{itemize}

To detect closed $k$-networks we shall apply the following simple lemma.

\begin{lemma}\label{lm1}
A family $\mathcal{K}$ of closed subsets of compact space $X$ is a $k$-network in $X$ if and only if for each point 
$x\in X$ and its neighborhood $U\subseteq X$ there is a finite subfamily $\F\subseteq\K$ such that $x\in\Int(\bigcup\F)\subseteq \bigcup\F\subseteq U$ and $x\in \cl(\Int(F))$ for all $F\in\F$.
\end{lemma}

\begin{proof} The ``if'' part is trivial. To prove the ``only if'' part, assume that $\mathcal{K}$ is a $k$-network in $X$. Let us fix a point $x\in X$ and its neighbourhood $U$.  By the regularity of $X$ there is an open set $V$ such that $x\in V\subseteq \cl V\subseteq U$. The family $\K$, being a  $k$-network, contains a finite subfamily $\mathcal{F}\subseteq\mathcal{K}$ such that $\cl V\subseteq \bigcup \mathcal{F}\subseteq U$. In the family $\F$ consider the subfamily $\F_x=\{F\in\F:x\in\cl(\Int(F))\}$ and observe that the set $V_x=V\setminus\bigcup_{F\in\F\setminus\F_x}\cl(\Int(F))$ is an open neighborhood of $x$. It remains to prove that $V_x\subset \bigcup\F_x$. Assuming the opposite, we would conclude that the open set $W=V_x\setminus\bigcup\F_x$ is non-empty. Since $W\subset V_x\subset V\subset \bigcup\F$, for some set $F\in\F\setminus \F_x$ the intersection $F\cap W$ has non-empty interior in $W$, which implies that the sets $W\subset V_x$ meet the interior of $F$. But this contradicts the choice of the set $V_x$.
\end{proof}

The following characterization of supercompactness is well-known for specialists.
We present the proof for convenience of the reader.

\begin{theorem}\label{k-net}
A compact Hausdorff space $X$ is supercompact if and only if it possesses a binary closed $k$-network.
\end{theorem}

\begin{proof} To prove the ``if'' part, assume that a compact space $X$ possesses a binary closed $k$-network $\mathcal{K}$. Then binary property of $\K$ implies that the family $\{X\setminus K\colon K\in\mathcal{K}\}$ is open subbase in $X$ witnessing that the compact space $X$ is supercompact.

To prove the ``only if'' part, assume that the space $X$ is supercompact. Then there is a subbase $\mathcal B$ of the topology of $X$ such that each cover of $X$ by elements of the subbase contains a two-element subcover. Without loss of generality, the family $\mathcal B$ is stable under unions. It follows that the family $\K=\{X\setminus U:U\in\mathcal B\}$ is a  binary subbase for closed sets, stable under intersections. It remains to check that this family is a $k$-network for $X$. Fix a compact set $K$ in $X$ and an open neighborhood $U$ of $K$ in $X$. By the compactness there is a finite family $\{A_1,\dots,A_l\}$ consisting of finite unions $A_i$ of elements of $\mathcal{K}$, such that $K\subseteq A_1\cap\dots\cap A_l\subseteq U$. It remains observe that the set  $A_1\cap\dots\cap A_l$ is a finite union of elements of $\mathcal{K}$ because $\mathcal{K}$ is stable under intersections.
\end{proof}

Supercompact spaces have the following property established by M.~Bell \cite{Bell}.

\begin{theorem}[Bell]\label{bell} If a space $X$ is a continuous image of a supercompact space, then for each open dense subset $U\subseteq X$ the complement $X\setminus U$ has density $d(X\setminus U)\le nw(U)\le|U|$.
\end{theorem}

In this theorem $nw(U)$ stands for the network weight of $U$.
The best known positive result on hereditarily supercompact spaces can be derived combining the results of Bula, Nikiel, Tuncali, Tymchatyn \cite{BNTT} and Rudin \cite{Rudin} on monotonically normal compact spaces.

We recall that a topological space $X$ is called {\em monotonically normal} if there is an operator $G$ assigning to each pair $(A,B)$ of disjoint closed subsets of $X$ an open set $G(A,B)\subset X$ such that
\begin{itemize}
\item $A\subset G(A,B)\subseteq \cl G(A,B)\subseteq X\setminus B$;
\item $G(A,B)\subseteq G(A',B')$ for any pairs $(A,B)$, $(A',B')$ of disjoint closed sets with $A\subseteq A'$ and $B'\subseteq B$.
\end{itemize}
By a deep result of M.~E.~Rudin \cite{Rudin}, a compact space is monotonically normal if and only if it is a continuous image of a linearly ordered compact space. Another deep result of Bula, Nikiel, Tuncali and Tymchatyn \cite{BNTT} says that each continuous image of a linearly ordered compact space is supercompact. Combining these two results, we get the following theorem.

\begin{theorem}\label{monotone} Each monotonically normal compact space is hereditarily supercompact.\qed
\end{theorem}

It is known (see \cite{Ost}, \cite{Rudin}) that each monotonically normal space $X$ is {\em sub-hereditarily separable} in the sense that each separable subspace of $X$ is hereditarily separable. Sub-hereditarily separable spaces were introduced and studied by D.~Daniel and M.~Tuncali in \cite{DT}.
The sub-hereditary separability of monotonically normal spaces and Theorem~\ref{monotone} motivate the following:

\begin{problem} Is each hereditarily supercompact space sub-hereditarily separable?
\end{problem}

We shall give an affirmative answer to this problem under the Set-Theoretic assumption $\w_1<\mathfrak p$. Here $\mathfrak p$ is the smallest cardinality of a base of a free filter $\F$ on a countable set $X$, which has no infinite pseudo-intersection. An infinite subset $I\subseteq X$ is called a {\em pseudo-intersection} of the filter $\F$ is $I\setminus F$ is finite for each set $F\in\F$.
The following lemma follows directly from the definition of the cardinal $\mathfrak p$, see \cite[\S6]{vD}.

\begin{lemma}\label{lem:p} Let  $A$ be a countable subset of a topological space $X$. If a point  $a\in\cl A$ has character $\chi(a,X)<\mathfrak p$ in $X$ then $a$ is the limit of some sequence $\{a_n\}_{n\in\w}\subseteq A$.
\end{lemma}

\begin{theorem}\label{shs=hs} Under $\omega_1<\mathfrak p$ each hereditarily supercompact space is sub-hereditarily separable.
\end{theorem}

\begin{proof} Given a hereditarily supercompact space $X$, we need to check that each separable subspace $Z$ of $X$ is hereditarily separable. Assume conversely that $Z$ is not hereditarily separable.
Then $Z$ contains a transfinite sequence $(z_\alpha)_{\alpha<\w_1}$ which is {\em left separated} in the sense that for every ordinal $\alpha<\w_1$ the point $z_\alpha$ does not belong to the closure $\cl Z_\alpha$ of the set $Z_\alpha=\{z_\beta\}_{\beta<\alpha}$ in $X$. Consider the closure $\cl Z$ of $Z$ in $X$ and take any map $f\colon \cl Z\to Y$ onto a compact space $Y$ of weight $\w_1$ such that $f(z_\alpha)\notin f(\cl Z_\alpha)$ for all $\alpha<\w_1$.
It follows that the transfinite sequence of points $y_\alpha=f(z_\alpha)$, $\alpha<\w_1$, is left separated and hence the space $Y$ is not hereditarily separable.
On the other hand, the space $Y$ is separable (being the continuous image of the separable space $\cl Z$). So, we can find a countable dense subset $Q\subseteq Y$. Since the compact space $Y$ is not hereditarily separable, it contains a closed non-separable subspace $F\subseteq Y$ according to \cite[3.12.9(d)]{En}. Replacing $F$ by the closure of the set $F\setminus \cl(F\cap Q)$, we can assume that the set $F\cap Q$ is nowhere dense in $F$ and in $Y$. Then the countable set $Q\setminus F$ is dense in $Y$. Replacing $Q$ by $Q\setminus F$, we can assume that $F\cap Q=\emptyset$.

Let $\exp(Y)$ be the space of all non-empty closed subsets of $Y$, endowed with the Vietoris topology.
The density of $Q$ in $Y$ implies that the countable set $[Q]^{<\w}$ of all non-empty finite subsets of $Q$ is dense in the hyperspace $\exp(Y)$. Since $\exp(Y)$ has weight $\w_1<\mathfrak p$, Lemma~\ref{lem:p} yields a sequence $F_n\in[Q]^{<\w}$, $n\in\w$, of finite subsets of $Q$, which converges to $F$ in the Vietoris topology of $\exp(Y)$. Then $K=F\cup\bigcup_{n\in\w}F_n$ is a compact subset of $Y$ and $\bigcup_{n\in\w}F_n=K\setminus F$ is a countable dense open subset of $K$.

Since the space $X$ is hereditarily supercompact, the preimage $f^{-1}(K)\subseteq \cl Z$ is supercompact and hence $K$ is a continuous image of a supercompact space. By Bell's Theorem~\ref{bell}, the closed subset $F=K\setminus \bigcup_{n\in\w}F_n$ is separable, which is a desired contradiction.
\end{proof}

\begin{corollary}\label{c:MACH} Under \textup{(MA$+\neg$CH)} each separable hereditarily supercompact space is hereditarily separable and hereditarily Lindel\"of.
\end{corollary}

\begin{proof} It is well-known that (MA+$\neg$CH) implies $\w_1<\mathfrak p=\mathfrak c$. By Theorem~\ref{shs=hs}, each separable hereditarily supercompact space $X$ is hereditarily separable. By a result of Szentmikl\'ossy \cite{Szent} (see also \cite[6.4]{Roit}), under (MA+$\neg$CH) each compact hereditarily separable space is hereditarily Lindel\"of. So, $X$ is hereditarily Lindel\"of.
\end{proof}

In the class of dyadic compacta the hereditary supercompactness is equivalent to the metrizability.
Let us recall that a compact space $X$ is {\em dyadic} if $X$ is a continuous image of the Cantor cube $\{0,1\}^\kappa$ for some cardinal $\kappa$.

\begin{theorem}\label{dyadic} A dyadic compact space is metrizable if and only if it is hereditarily supercompact.
\end{theorem}

\begin{proof} By a result of Gerlits \cite{Ger} or Efimov \cite{Efim}, each non-metrizable dyadic compact space $X$ contains a topological copy of the Cantor cube $2^{\w_1}$. By~\cite[Sec.~1.1]{JvM}, there exists a scattered compact space $Z$ of weight $\w_1$, which is not supercompact. Being zero-dimensional, the space $Z$ embeds into the Cantor cube $2^{\w_1}$ and hence embeds into the dyadic compact space $X$. Now we see that the space $X$ contains a non-supercompact closed subspace  and hence is not hereditarily supercompact.
\end{proof}

Now we present one construction preserving hereditarily supercompact spaces.

\begin{proposition}\label{p1} For any (hereditarily) supercompact spaces $X_i$, $i\in I$, the one-point compactification $\alpha X$ of the topological sum $X=\bigoplus_{i\in I}X_i$ is (hereditarily) supercompact.
\end{proposition}

\begin{proof} For every $i\in I$, fix a binary closed $k$-network $\K_i$ on the space $X_i$.
Then $$\K=\big\{\alpha X\setminus \textstyle{\bigoplus\limits_{i\in F}}X_i:F\in[I]^{<\w}\big\}\cup \bigcup_{i\in I}\K_i $$is a binary closed $k$-network witnessing that the space $\alpha X$ is supercompact.

Now assuming that each space $X_i$, $i\in I$, is hereditarily supercompact, we shall show that so is the space $\alpha X$. Fix any closed subset $Y\subset\alpha X$ and for every $i\in I$ consider the supercompact space $Y_i=Y\cap X_i$. If the compactifying point $\infty$ of $\alpha X$ does not belong to $Y$,  then all but finitely many spaces $Y_i$ are empty and $Y$ is supercompact, being a disjoint union of a finite family of supercompact spaces. If $\infty\in Y$, then $Y$ is supercompact, being the one-point compactification of the topological sum $\bigoplus_{i\in I}Y_i$ of supercompact spaces $Y_i$, $i\in I$.
\end{proof}

\begin{corollary}\label{onepoint} The one-point compactification $\alpha D$ of any discrete space $D$ is hereditarily supercompact.
\end{corollary}

A topological space $X$ is called {\em hereditarily paracompact} if each subspace of $X$ is paracompact.
In \cite{BL} it was proved that the class of scattered compact hereditarily paracompact spaces coincides with the smallest topological class $\A$ which contains the singleton and is closed with respect to taking the one-point compactification of a topological sum $\bigoplus_{i\in I}X_i$ of spaces from the class $\A$. This characterization combined with Proposition~\ref{p1} implies:

\begin{corollary} Each scattered compact hereditarily paracompact space is hereditarily supercompact.
\end{corollary}

The characterization of scattered hereditarily paracompact spaces given in \cite{BL} implies the following metrizability criterion:

\begin{corollary} A scattered compact space is metrizable if and only if it is separable and hereditarily paracompact.
\end{corollary}

The preceding two corollaries motivate the following problem.

\begin{problem} Is each separable scattered hereditarily supercompact space metrizable?
\end{problem}

This problem has an affirmative solution under (MA+$\neg$CH), which follows from Corollary~\ref{c:MACH} and the metrizability of hereditarily Lindel\"of (and hence first countable) scattered compact spaces \cite{Sema}:

\begin{theorem}\label{scat-met} Under \textup{(MA+$\neg$CH)} a scattered compact space is metrizable if and only if it is separable and hereditarily supercompact.
\end{theorem}

In ZFC we can prove the metrizability of scattered supercompact spaces which have countable scattered height. The {\em scattered height} of a scattered space $X$ is defined as follows. For a subset $A\subseteq X$ denote by $A^{(1)}$ the set of all non-isolated points of $A$. Let $X^{(0)}=X$ and for each ordinal $\alpha$ define the $\alpha$-th derived set $X^{(\alpha)}$ of $X$ by the recursive formula
$$X^{(\alpha)}=\bigcap_{\beta<\alpha}(X^{(\beta)})^{(1)}.$$
The {\em scattered height} of $X$ is the smallest ordinal $\alpha$ for which the $\alpha$-th derived set $X^{(\alpha)}$ is empty.

For a topological space $X$ by $[X]^{<\w}$ we denote the space of all non-empty finite subsets of $X$ endowed with the Vietoris topology. It is well-known that for a metrizable space $X$ the hyperspace $[X]^{<\w}$ is metrizable. We recall that a topological space $X$ has {\em countable tightness} if for any subset $A\subseteq X$ and a point $a\in\cl A$ in its closure there is a countable subset $B\subseteq A$ with $a\in\cl B$.

\begin{proposition}\label{scat} A scattered compact space $X$ is metrizable if and only if
\begin{enumerate}
\item $X$ is separable,
\item $X$ has countable scattered height or $[X]^{<\w}$ has countable tightness,
\item $X$ is a continuous image of a supercompact space.
\end{enumerate}
\end{proposition}

\begin{proof} If a scattered compact space $X$ is metrizable, then it is countable and hence is separable and has countable scattered height. Moreover, the hyperspace $\exp(X)$ is metrizable and hence $[X]^{<\w}\subseteq \exp(X)$ is metrizable and has countable tightness. By Strok-Szymanski Theorem \cite{SS}, the metrizable compact space $X$ is (hereditarily) supercompact.

Now assume that a separable scattered space $X$ is a continuous image of a supercompact space and $X$ has countable scattered height or $[X]^{<\w}$ has countable tightness.

We shall consider two cases separately.
First assume that the space $[X]^{<\w}$ has countable tightness. In this case we can apply Theorem of van Mill and Mills \cite{vMM} to conclude that the separable space $X$ if first countable. By \cite{Sema}, the first countable compact scattered space $X$ is metrizable.

Now assume that the scattered space $X$ has countable scattered height $\alpha$.
By transfinite induction we shall prove that for each countable ordinal $\beta\le \alpha$ the set $X\setminus X^{(\beta)}$ is  countable. This is trivial for $\beta=0$ and follows from the separability of $X$ for $\beta=1$. Assume that for some non-zero ordinal $\beta\le\alpha$ we have proved that the set $X\setminus X^{(\gamma)}$ is countable for all ordinals $\gamma<\beta$. If $\beta$ is a limit ordinal, then $$X\setminus X^{(\beta)}=\bigcup_{\gamma<\beta}X\setminus X^{(\gamma)}$$is countable as the countable union of countable sets.

Now assume that $\beta=\gamma+1$ is a successor ordinal. By the inductive assumption, the set $X\setminus X^{(\gamma)}$ is countable. Since the space $X$ is a continuous image of a supercompact space, and $U=X\setminus X^{(\gamma)}$ is an open dense countable subset of $X$, we get $d(X^{(\gamma)})\le nw(U)\le|U|$ according to Bell's Theorem~\ref{bell}. This implies that the set $X^{(\gamma)}\setminus X^{(\beta)}$ of isolated points of the scattered space $X^{(\gamma)}$ is countable and hence the set $X\setminus X^{(\beta)}=(X\setminus X^{(\gamma)})\cup(X^{(\gamma)}\setminus X^{(\beta)})$ is countable, which completes the inductive step.

At the final $\alpha$-th step of induction, we get that the space $X\setminus X^{(\alpha)}=X\setminus\emptyset=X$ is countable. Being a compact countable space, $X$ is metrizable.
\end{proof}

Proposition~\ref{scat} implies another metrization criterion of scattered compact spaces.

\begin{corollary}\label{corher} A scattered compact space $X$ is metrizable if and only if
$X$ is a continuous image of a supercompact space and all finite powers $X^n$, $n\in\w$, of $X$ are hereditarily separable.
\end{corollary}

\begin{proof} The ``only if'' part follows the supercompactness of compact metrizable  spaces. To prove the ``if'' part, assume that a compact scattered space $X$ is a continuous image of a supercompact space and each finite power $X^n$, $n\in\w$, of $X$ is hereditarily separable. Then the space $X$ is separable and the space $[X]^{<\w}$ is hereditarily separable, being the continuous image of the topological sum $\bigoplus_{n\in\w}X^n$ of hereditarily separable spaces $X^n$, $n\in\w$. Being hereditarily separable, the space $[X]^{<\w}$ has countable tightness. By Theorem~\ref{scat}, the scattered compact space $X$ is metrizable.
\end{proof}

Now we describe another construction preserving hereditarily supercompact spaces.

	\begin{proposition}\label{ext}  A compact topological space $X$ is (hereditarily) supercompact if $X$ contains a closed subspace $Z\subseteq X$ such that
	\begin{enumerate}
	\item $Z$ is (hereditarily) supercompact;
	\item $X\setminus Z$ is discrete;
	\item $Z$ is a retract in $X$.\label{ppp}
	\end{enumerate}
	\end{proposition}

\begin{proof} It suffices to prove that a closed subspace $Y$ of $X$ is supercompact provided the intersection $Y\cap Z$ is supercompact. By Theorem~\ref{k-net}, the supercompact space $Y\cap Z$ possesses a binary closed $k$-network $\K$. By our assumption, there is a retraction $r\colon X\to Z$.  The interested reader
can check that the family
$$\K_Y=\{r^{-1}(K)\cap Y\setminus F:K\in\K,\;F\in[X\setminus Z]^{<\w}\}\cup\{\{y\}:y\in Y\setminus Z\}$$is a binary closed $k$-network witnessing the supercompactness of the space $Y$.
\end{proof}

\begin{remark} The last condition is essential in Proposition~\ref{ext}. Indeed, take any separable scattered compact space $X$ whose set $X'$ of non-isolated points is uncountable and has a unique non-isolated point. By Proposition~\ref{p1}, the space $X'$ is supercompact, being the one-point compactification of a discrete space. On the other hand, the scattered space $X$ has finite scattered height, is separable and non-metrizable. By Proposition~\ref{scat}, $X$ is not supercompact.
\end{remark}

\section{The hereditary supercompactness of spaces $K\times\alpha D$}

In this section we study the problem of hereditary supercompactness of products $K\times\alpha D$ where $K$ is a metrizable compact space $K$ and $\alpha D=\{\infty\}\cup D$ is the one-point compactification of a discrete space $D$. Here $\infty\notin D$ is the compactifying point of $\alpha D$. The topology of the space $\alpha D$ consists of all subsets $U\subset\alpha D$ such that either $\infty\notin U$ or $\alpha\setminus U$ is finite. So, for a finite $D$ the space $\alpha D$ also is finite and has cardinality $|\alpha D|=|D|+1$.

By Theorem~\ref{monotone} and Corollary~\ref{onepoint}, the spaces $K$ and $\alpha D$ both are hereditarily supercompact. Moreover, the spaces $K$, $\alpha D$ belong to many important classes of compact Hausdorff spaces. In particular, they are (uniform) Eberlein and Rosenthal compact.

Let us recall that a compact space $X$ is called
\begin{itemize}
\item ({\em uniform}) {\em Eberlein compact} if $X$ embeds into a (Hilbert) Banach space endowed with the weak topology;
\item {\em Rosenthal compact} if $X$ embeds into the space $B_1(P)\subseteq\IR^P$ of functions of the first Baire class on a Polish space $P$.
\end{itemize}

It is clear (and well-known) that the classes of (uniform) Eberlein compact and Rosenthal compact spaces are closed with respect to taking closed subspaces and countable products. Any Rosenthal compact space $K$ has cardinality $|K|\le\mathfrak c$ not less than the cardinality of continuum $\mathfrak c$.
The following proposition is easy (and well-known) and is written in sake of references in the subsequent text.

\begin{proposition}\label{p3.1} For any compact metrizable space $K$ and any discrete space $D$ of cardinality $|D|\le\mathfrak c$,  the space $K\times\alpha D$ is uniform Eberlein compact and Rosenthal compact.
Consequently, each closed subspace $X\subset K\times\alpha D$ is uniform Eberlein and Rosenthal compact.
\end{proposition}

Now we establish some less obvious properties of closed subspaces of products $K\times\alpha D$.

We shall say that an indexed family $\{X_i\}_{i\in D}$ of subsets of a space $K$ is
\begin{itemize}
\item {\em point-countable} if for each point $x\in K$ the set $\{i\in D\colon x\in X_i\}$ is at most countable;
\item {\em nwd-countable} if for each nowhere dense subset $F\subseteq K$ the family $\{i\in D\colon X_i\cap F\ne\emptyset\}$ is at most countable;
\item {\em weakly nwd-countable} if for each closed nowhere dense subset $F\subseteq K$ the set $\{i\in D\colon X_i\cap F$ is not open in $X_i\}$ is at most countable;
\end{itemize}
These notions relate as follows:
$$\mbox{point-countable $\Leftarrow$ nwd-countable $\Rightarrow$ weakly nwd-countable}.
$$

For a closed subset $X\subseteq K\times\alpha D$ and a point $i\in\alpha D$ let $X_i=\{x\in K\colon (x,i)\in X\}$ be the $i$-th section of the set $X$. It is clear that $X=\bigcup_{i\in\alpha D}X_i\times\{i\}$.

\begin{proposition}\label{pointc}
Let $K$ be a compact metrizable space and $\alpha D$ be the one-point compactification of a discrete space $D$. A closed subspace $X\subseteq K\times\alpha D$ is first countable if and only if the indexed family $\{X_i:i\in D\}$ is point-countable.
\end{proposition}

\begin{proof} If the family $\{X_i:i\in D\}$ is not point countable, then for some point $x\in K$ the set $E=\{i\in D:x\in X_i\}$ is uncountable. In this case $X$ contains the space $\{x\}\times \alpha E$ and hence cannot be first countable.

Now assume that the family $\{X_i:i\in D\}$ is point countable. We need to show that the space $X$ is first countable at each point $(x,i)\in X$. This is clear if $i\ne\infty$. So, we assume that $i=\infty$. By our assumption the set $E=\{i\in D:x\in X_i\}$ is at most countable. Fix a countable base $\mathcal B_x$ of open neighborhoods of the point $x$ in the metrizable space $K$ and observe that the singleton $$\{x\}=\bigcap_{e\in E}\bigcap_{B\in\mathcal B_x}B\times(\alpha D\setminus\{e\})$$is a $G_\delta$-set in $X$. By the compactness of $X$, the space $X$ is first countable at $x$.
\end{proof}

\begin{theorem}\label{para} For a closed subspace $X$ of the product $K\times\alpha D$ of a metrizable compact space $K$ and the one-point compactification $\alpha D$ of a discrete space $D$, the following conditions are equivalent:
\begin{enumerate}
\item $X$ is hereditarily paracompact;
\item  $X$ is hereditarily normal;
\item the indexed family $\{X_i\colon i\in D\}$ is weakly nwd-countable.
\end{enumerate}
\end{theorem}

\begin{proof}
The implication (1)$\Rightarrow$(2) is obvious.
\smallskip

$(2)\Ra(3)$ Assume that $X$ is hereditarily normal but the family $\{X_i\colon i\in D\}$ is not weakly nwd-countable. In this case we can find a closed nowhere subset $F\subseteq K$ such that the set $\{i\in D\colon X_i\cap F\text{ is not open in }X_i\}$ is uncountable. We claim that the open subspace $Y=X\setminus (F\times\{\infty\})$ of $X$ is not normal. Indeed, consider the disjoint closed subsets $Y\cap ((K\setminus F)\times\{\infty\})$ and $Y\cap(F\times D)$ in the space $Y$. Assuming that the space $Y$ is normal, we would find disjoint open sets $U,V\subset Y$ such that $Y\cap ((K\setminus F)\times\{\infty\})\subset U$ and $Y\cap (F\times D)\subset V$. Using the Lindel\"of property of the metrizable separable space $Y\cap ((K\setminus F)\times\{\infty\})\subset U$, we can find a subset $\widetilde D\subset D$ with at most countable complement $\widetilde D\setminus D$ such that $Y\cap ((K\setminus F)\times \widetilde D)\subset U$. By the choice of $F$, there is a point $i\in \widetilde D$ such that the set $X_i\cap F$ is not open in $X_i$. Since the set $V\cap (X_i\times\{i\})$ is an open neighborhood of the non-open set $Y\cap (F\cap\{i\})=(X_i\cap F)\times\{i\}$ in $X_i\times\{i\}$, there is a point $(x,i)\in V\cap ((X_i\setminus F)\times\{i\})$. Then $(x,i)\in
Y\cap ((K\setminus F)\times \tilde D)\subset U$, which is not possible as $V$ and $U$ are disjoint.
\smallskip

$(3)\Ra(1)$ Assuming that the indexed family $\{X_i\colon i\in D\}$ is weakly nwd-countable in $K$, we shall prove that the space $X$ is hereditarily paracompact. This will follow as soon as we check that  each open subspace $U$ of $X$ is paracompact. Given any open cover $\U$ of $U$, we need to construct a locally finite open cover $\V$ of $U$ refining the cover $\U$. For every $i\in\alpha D$ consider the section $U_i=\{x\in K: (x,i)\in U\}$ of the set $U$. It follows that the set $U_\infty$ is open in the section $X_\infty=\{x\in K:(x,\infty)\in X\}$ of $X$. So, its boundary $\partial U_\infty$ in $X_\infty$ is nowhere dense in $K$. The weak nwd-countability of the family $\{X_i\}_{i\in D}$ implies that the set $E_\infty=\{i\in D:X_i\cap\partial U_\infty$ is not open in $X_i\}$ is at most countable.

By the definition of the set $E_\infty$, for every index $i\in D\setminus E_\infty$ the set $X_i\cap \partial U_\infty$ is open in $X_i$ and hence the set $U_i\setminus U_\infty$ is open in $U_i$.  By the paracompactness of the metrizable space $U_i\setminus U_\infty$, we can find a locally finite open cover $\U_i$ of $(U_i\setminus U_\infty)\times\{i\}=\bigcup\U_i$ inscribed into the cover $\U$.

Being locally compact and $\sigma$-compact, the set $U_{\infty}$ can be written as the countable union $U_\infty=\bigcup_{n\in\w}V_n$ of open subsets such that $V_1=\emptyset$ and for every $n\in\w$ the closure $\cl V_n$ of $V_n$ in $U_\infty$ is compact and lies in the set $V_{n+1}$. For every $n\in\w$ choose a finite family $\V_n$ of open subsets of $U$ such that $$(\cl V_{n+1}\setminus V_{n-1})\times\{\infty\}\subset \textstyle{\bigcup}\V_n\subset (V_{n+2}\setminus\cl V_{n-2})\times\alpha D$$ and the family $\V_n$ is inscribed into the cover $\U$. Using the compactness of $\cl V_{n+1}\setminus V_{n-1}$ we can find a finite subset $E_n\subset D$ such that $U\cap((\cl V_{n+1}\setminus V_{n-1})\times (\alpha D\setminus E_n))\subset\bigcup \V_n$. Replacing each set $E_n$ by a larger finite set, if necessary, we can assume that $E_{n-1}\subset E_{n}$ and $E_\infty\subset \bigcup_{n\in\w}E_n$. Since $V_1=\emptyset$, we can put $E_0=\emptyset$.

For every $n\in\w$, by the paracompactness of the open metrizable subspace
$M_n=U\cap ((K\setminus \cl V_n)\times (E_{n+1}\setminus E_n))$ of $U$, there is a locally finite open cover $\W_n$ of $M_n=\cup\W_n$ refining the cover $\U$.  The interested reader can check that the family $$\V=\bigcup_{i\in D\setminus E_\infty}\U_i\cup\bigcup_{n\in\w}(\V_n\cup\W_n)$$ is a locally finite open cover of $U$ refining the cover $\U$ and witnessing that the space $U$ is paracompact.
\end{proof}

Next we study the supercompactness of closed subspaces $X\subseteq K\times\alpha D$.

\begin{proposition} Let $K$ be a compact metrizable space and $\alpha D=\{\infty\}\cup D$ be the one-point compactification of a discrete space $D$. A closed subspace $X\subseteq K\times\alpha D$ is supercompact if the subspace $X_\infty=\{x\in K:(x,\infty)\in X\}$ is zero-dimensional.
\end{proposition}

\begin{proof} Fix a metric $d$ generating the topology of $K$. Multiplying the metric $d$ by a suitable positive constant, we can assume that $\diam(K,d)\le 1$. Let $\U_0=\{K\}$. Using the zero-dimensionality of the space $X_\infty$, by induction for every $n\in\w$ we can construct a finite disjoint family $\mathcal B_n$ of closed subsets of $K$ such that
\begin{enumerate}
\item $X_\infty\subset\bigcup_{B\in\mathcal B_n}\Int(B)$;
\item each set $B\in\mathcal B_n$ has diameter $\diam(B)\le 2^{-n}$;
\item the family $\mathcal B_n$ refines the family $\{\Int(B):B\in\mathcal B_{n-1}\}$;
\item for each set $B\in\mathcal B_n$ the intersection $B\cap X_\infty$ is a non-empty closed-and-open subset of $X_\infty$.
\end{enumerate}
For every $n\in\w$ consider the set $D_n=\{i\in D:X_i\subset \bigcup\mathcal B_n\}$ and observe that the complement $D\setminus D_n$ is finite and $D_{n+1}\subset D_{n}$ for all $n\in\w$. For every index $i\in D_{n}\setminus D_{n+1}$ the compact metrizable space $X_i$ is supercompact and hence possesses a binary closed $k$-network $\mathcal K_i$. Replacing $\K_i$ by its subfamily $\{K\in\K_i:\exists B\in\mathcal B_n\;\;K\subset B\}$ we can assume that the family $\K_i$ refines the family $\mathcal B_n$.

The reader can check that the family
$$
\begin{aligned}
\K=\;&\big\{X\cap(B\times \alpha E):n\in\w,\;B\in\mathcal B_n,\; E\subset D_n \mbox{ and } |D\setminus E|<\infty\big\}\cup\\
&\cup\big\{K\times\{i\}:n\in\w,\;\;i\in D_n\setminus D_{n+1},\;\;K\in\mathcal K_i\big\}\cup\\
&\cup\big\{X\cap(B\times \{i\}):B\in\bigcup_{n\in\w}\mathcal B_n,\; i\in \bigcap_{n\in\w}D_n\big\}
\end{aligned}
$$
is a binary closed $k$-network for the space $X$ witnessing that this space is supercompact.
\end{proof}

\begin{corollary}\label{c:zero} For each zero-dimensional compact metrizable space $K$ and each discrete space $D$ the product $K\times\alpha D$ is hereditarily supercompact.
\end{corollary}

If the space $K$ is infinite and $D$ is uncountable, then the product $K\times\alpha D$ is not monotonically normal \cite{HLZ}. So, the class of hereditarily supercompact spaces is wider than the class of monotonically normal compact spaces.

The zero-dimensionality of the space $K$ is essential in Corollary~\ref{c:zero}. We shall show that for the closed interval $\II=[0,1]$ and any discrete space $D$ of cardinality $|D|\ge\non(\M)$ the product $\II\times\alpha D$ is not hereditarily supercompact.

We shall need some information on the hyperspaces of compact metrizable spaces. For a compact space $K$ its {\em hyperspace} $\exp(K)$ is the space of all non-empty closed subsets of $K$ endowed with the Vietoris topology. If the topology of the space $K$ is generated by a metric $d$, then the Vietoris topology of the hyperspace $\exp(K)$ is generated by the Hausdorff metric
$$d_H(A,B)=\max\{\sup_{a\in A}d(a,B),\max_{b\in B}d(b,A)\}, \;\;A,B\in\exp(K).$$

Let us recall that a subset $A$ of a topological space $X$
is called
\begin{itemize}
\item {\em meager} in $X$ if it can be written as a countable union $A=\bigcup_{n\in\w}A_n$ of nowhere dense subsets of $X$;
\item {\em Baire} in $X$ if for each non-empty open subset $V\subseteq X$ the set $A\cap V$ is not meager in $V$.
\end{itemize}
It is well-known that for each non-meager subset $B$ in a Polish space $X$ there is a non-empty open subset $U\subseteq X$ such that the intersection $U\cap B$ is a dense Baire subspace in $U$.

We recall that for a closed subset $X\subseteq K\times\alpha D$ and a point $i\in\alpha D$ by $X_i=\{x\in K:(x,i)\in X\}$ we denote the $i$-th section of $X$.

\begin{theorem}\label{super} Let $X\subseteq\II\times\alpha D$ be a closed subset of the product of the closed unit interval $\II=[0,1]$ and the one-point compactification $\alpha D=\{\infty\}\cup D$ of a discrete space $D$. If the space $X$ is supercompact, then the subspace $\mathcal X=\{X_i:i\in D\}$ is meager in the hyperspace $\exp(\II)$.
\end{theorem}

\begin{proof} Assume conversely that the subspace $\mathcal X=\{X_i\colon i\in D\}$ is not meager in $\exp(\II)$.  By Theorem~\ref{k-net}, the supercompact space $X$ possesses a binary closed $k$-network $\K$. For each point $i\in D$ the section $X_i\times\{i\}=X\cap (\II\times\{i\})$ is closed-and-open in $X$. Since $\K$ is a $k$-network, there exists a finite subfamily $\K_i\subseteq\K$ such that $X_i\times\{i\}=\bigcup\K_i$. It follows that some set $K_i\in\K_i$ has non-empty interior in $X_i\times\{i\}$. So, we can find two rational numbers $a_i<b_i$ in $\II$ such that  $\emptyset\ne\big((a_i,b_i)\cap X_i\big)\times\{i\}\subseteq K_i$.

For any rational numbers $a<b$ in $\II$, consider the subset
$$D_{a,b}=\{i\in D:(a_i,b_i)=(a,b)\}$$ and observe that
$$D=\bigcup\{D_{a,b}:a<b,\;a,b\in\II\cap\IQ\}.$$ Since the space $\mathcal X=\{X_i:i\in D\}$ is not meager in $\exp(\II)$, for some rational numbers $a<b$ in $\II$ the space $\mathcal X_{a,b}=\{X_i:i\in D_{a,b}\}$ is not meager in $\exp(\II)$. Consequently, for some non-empty open set $\U\subseteq\exp_0(\II)$ the intersection $\mathcal X_{a,b}\cap\U$ is a dense Baire subspace in $\U$.

Choose a set $C_0\in \mathcal X_{a,b}\cap \U$, a point $c\in (a,b)\cap C_0$ and $\e>0$ such that $(c-\e,c+\e)\subset (a,b)$ and the open set $\U$ contains the $\e$-ball
$$\BB_\e(C_0)=\{C\in\exp(\II):d_H(C,C_0)<\e\}\subset\exp(\II)$$
centered in $C_0$. It follows that $\mathcal X_{a,b}\cap\BB_{\e}(C_0)$ is a dense Baire subspace in the ball $\BB_\e(C_0)$. Let $D_\e=\{i\in D_{a,b}:X_i\in \BB_\e(C_0)\}$ and observe that $\{X_i:i\in D_\e\}=\mathcal X_{a,b}\cap\BB_\e(C_0)$. It follows that $$(c,\infty)\in C_0\times\{\infty\}\in \cl(\{X_i\times\{i\}:i\in D_\e\}\subseteq X.$$

By Lemma~\ref{lm1}, the $k$-network $\K$ contains a finite subfamily $\F\subseteq\K$ such that $$(c,\infty)\in \Int(\textstyle{\bigcup}\F)\subseteq\textstyle{\bigcup}\F\subseteq X\cap\big((c-\e/2,c+\e/2)\times\alpha D\big)$$ and $(c,\infty)\in\cl(\Int(F))$ for every $F\in\F$.

For every $F\in\F$ and a real number $\mu\in [c-\e/2,c+\e/2]$ consider the subfamily
$$D_{F,\mu}=\big\{i\in D_\e:F\cap(X_i\times\{i\})\not\subseteq[\mu,1]\times \alpha D\big\}.$$
It follows from $F\subseteq(c-\e/2,c+\e/2)\times\alpha D$ that $D_{F,c-\e/2}=\emptyset$,
so it is legal to consider the real number
$$\mu_F=\sup\{\mu\in[c-\e/2,c+\e/2]:D_{F,\mu}\mbox{ is finite}\}.$$
It follows that the set $D_{F,\mu_F}$ is at most countable and $(\mu_F,\infty)\in \cl\big(\bigcup_{i\in D_\e}X_i\big)\subseteq X$.

Choose a positive number $\delta<\e/2$ such that $$\delta<\frac12\min\big\{|x-y|:x,y\in\{\mu_F:F\in\F\},\;x\ne y\big\}$$and
$(c,\infty)\in X\cap\big((c-\delta,c+\delta)\times \alpha D_c\big)\subseteq \Int(\bigcup\F)$ for some
set $D_c\subseteq D$ with finite complement $D\setminus D_c$.

For every $F\in\F$ consider the point $(\mu_F,\infty)\in X$ and its open neighborhood $U_F=X\cap\big((\mu_F-\delta,\mu_F+\delta)\times\alpha D\big)$. Since $\F$ is a $k$-network, there is a finite subfamily $\K_F\subset\K$ such that
$$(\mu_F,\infty)\in\Int(\bigcup\K_F)\subseteq\bigcup\K_F\subseteq U_F\text{ and }(\mu_F,\infty)\in \cl(\Int(K))\subseteq K$$ for each $K\in\K_F$.
Since $(\mu_F,\infty)\in\Int(\bigcup\K_F)$, we can find a positive real number $\delta_F<\delta$ and a subset $D_\F\subset D$ with finite complement $D\setminus D_F$ such that $X\cap\big((\mu_F-\delta_F,\mu_F+\delta_F)\times \alpha D_F\big)\subseteq\Int(\bigcup\K_F)$.

It follows that the set
$$D'=(D\setminus D_c)\cup \bigcup_{F\in\F}(D\setminus D_F) \cup D_{F,\mu_F}$$
is at most countable. Taking into account that $\{X_i\colon i\in D_\e\}=\mathcal X_{a,b}\cap\BB_\e(C_0)$ is a dense Baire subspace of the Polish space $\BB_\e(C_0)$, we conclude that the set $\{X_i\colon i\in D_\e\setminus D'\}$ is dense in the $\e$-ball $\BB_\e(C_0)$.

Choose a closed subset $C\subseteq \II$ such that
\begin{enumerate}
\item $d_H(C,C_0)<\e$;
\item $C\cap (c-\delta,c+\delta)\ne\emptyset$;
\item $C\cap(\mu_F-\delta_F,\mu_F)\ne\emptyset$ for each $F\in\F$;
\item $C\cap[\mu_F,\mu_F+\delta_F]=\emptyset$ for each $F\in\F$.
\end{enumerate}
The choice of the set $C$ is possible since $c\in C_0$ and $\{\mu_F\colon F\in\F\}\subseteq [c-\e/2,c+\e/2]$.
Since the set $\{X_i:i\in D_\e\setminus D'\}$ is dense in $\BB_\e(C_0)$, we can additionally assume that $C=X_i$ for some $i\in D_\e\setminus D'$. The definition of the set $D'\not\ni i$ guarantees that $i\in D_c\cup\bigcup_{F\in\F}D_F$.

The choice of the set $C=X_i$ guarantees that the intersection $X_i\cap (c-\delta,c+\delta)$ contains some point $c'$. Then $(c',i)\in X\cap\big((c-\delta,c+\delta)\times D_c)\subset\Int(\bigcup \F)\subseteq\bigcup\F$ implies that $(c',i)\in F$ for some $F\in\F$. Taking into account that $i\notin D_{F,\mu_F}$, we conclude that $(c',i)\in F\cap(X_i\times\{i\})\subset [\mu_F,1]\times\{i\}$.
The choice of the set $C=X_i$ guarantees that the intersection $X_i\cap (\mu_F,\mu-\delta_F)$ contains some point $c_F$ while the intersection $X_i\cap[\mu_F,\mu_F+\delta_F)$ is empty. Then
$$(c_F,i)\in X\cap\big((\mu_F-\delta_F,\mu_F+\delta_F)\times D_F)\subseteq\Int(\textstyle{\bigcup}\K_F)\subseteq\textstyle{\bigcup}\K_F$$implies that $(c_F,i)\in K$ for some $K\in\K_F\subseteq\F$.

We claim that the subfamily $\mathcal L=\{K_i,K_F,F\}\subseteq\K$ is linked but has empty intersection. We recall that $K_i$ is a set from the $k$-network $\K$ such that $\big((a,b)\cap X_i\big)\times\{i\}\subset K_i\subset X_i\times\{i\}$. Observe that
\begin{itemize}
\item $F\cap K_F\ni (\mu_F,\infty)$,
\item $K_i\cap F\supseteq\big(\big(X_i\cap(a,b)\big)\times\{i\}\big)\cap F\ni (c',i)$, and
\item $K_i\cap K_F\supseteq \big(\big(X_i\cap(a,b)\big)\times\{i\}\big)\cap K_F\ni (c_F,i)$,
\end{itemize}
so the family $\mathcal L$ is linked.

On the other hand, taking into account that $F\cap(X_i\times\{i\})\subset[\mu_F,1)\times\{i\}$ and $X_i\cap[\mu_F,\mu_F+\delta_F)=\emptyset$, we conclude that the intersection
$$
\begin{aligned}
K_F\cap F\cap K_i&\subset\big((\mu_F-\delta_F,\mu_F+\delta_F)\times\alpha D\big)\cap
F\cap (X_i\times\{i\})\subset\\
&\subset \big((\mu_F-\delta_F,\mu_F+\delta_F)\cap [\mu_F,1]\cap X_i\big)\times\{i\}=\\
&=\big([\mu_F,\mu_F+\delta_F)\cap X_i)\times\{i\}=\emptyset\times\{i\}=\emptyset
\end{aligned}$$is empty. Therefore, the $k$-network $\K$ is not binary as it contains a linked subfamily $\mathcal L$ with empty intersection. This contradiction completes the proof.
\end{proof}

\section{Constructing non-meager nwd-countable families}

Theorems~\ref{para} and \ref{super} suggest the following open problem.

\begin{problem}\label{meg} Under which conditions is there a (weakly) nwd-countable non-meager family $\C\subseteq\exp(\II)$?
\end{problem}

We shall construct such a family under the set-theoretic assumption $\cof(\M)=\w_1$.
Here $\cof(\M)$ stands for the cofinality of the ideal $\M$ of meager subsets of the real line. It is known that $\cof(\M)=\mathfrak c$ under Martin's Axiom but the strict inequality $\w_1=\cof(\M)<\mathfrak c$ is consistent with ZFC, see~\cite{BJ} or \cite{Blass}. A Polish space $X$ is called {\em crowded} if its has no isolated points. In this case it has cardinality $\mathfrak c$ of continuum.

\begin{theorem}\label{disjoint} The hyperspace $\exp(X)$ of any crowded Polish space $X$ contains:
\begin{enumerate}
\item  a non-meager subfamily $\F$ of cardinality $\non(\M)$;
\item a non-meager nwd-countable subfamily $\mathcal D$ of cardinality $\w_1$ under assumption $\cof(\M)=\w_1$.
\end{enumerate}
\end{theorem}

\begin{proof}
1. Let $X$ be a crowded Polish space. Then its hyperspace $\exp(X)$ also is a crowded Polish space and hence it contains a dense topological copy $G$ of the Baire space $\w^\w$. By the definition of the cardinal $\non(\M)$, there is a non-meager subset $\C\subseteq G$ of cardinality $|\C|=\non(\M)$. Since $G$ is a dense $G_\delta$-set in $\exp(X)$, the set $\C$ is  non-meager in $\exp(X)$.
\smallskip

2. Now assume that $\cof(\M)=\w_1$. Then the ideals of meager subsets of the spaces $X$ and $\exp(X)$ both have cofinality equal to $\w_1$. So, we can fix a cofinal family $\{M_\alpha\}_{\alpha\in\w_1}$ of meager $F_\sigma$-subsets in $X$ and a cofinal family $\{\M_\alpha\}_{\alpha\in\w_1}$ of meager subsets in the hyperspace $\exp(X)$. We lose no generality assuming that $M_\alpha\subset M_\beta$ and $\M_\alpha\subset\M_\beta$ for all $\alpha<\beta<\w_1$. For each $\alpha\in\w$ consider the dense $G_\delta$-set $X\setminus M_\alpha$ and its hyperspace $\exp(X\setminus M_\alpha)$, which is a dense $G_\delta$-set in $\exp(X)$ and hence is not contained in the meager set $\M_\alpha$. So, we can choose a compact set $K_\alpha\in\exp(X\setminus M_\alpha)\setminus \M_\alpha$. To finish the proof it remains to observe that the family $\mathcal D=\{K_\alpha\}_{\alpha\in\w_1}$ is non-meager in $\exp(X)$ and nwd-countable in $X$.
\end{proof}

It should be mentioned that the existence of an uncountable nwd-countable family in $\exp(X)$ cannot be proved in ZFC. By $\add(\M)$ we denote the smallest cardinality $|\A|$ of a family $\A$ of meager subsets of $\IR$ with non-meager union $\bigcup\A$. It is clear that $\w_1\le\add(\M)\le\cof(\M)\le\mathfrak c$. Martin's Axiom implies that $\add(\M)=\mathfrak c$, see \cite{BJ} or \cite{Blass}.

\begin{proposition}\label{nwd-cc} If $\add(\M)>\w_1$, then each nwd-countable family $\C\subset\exp(X)$ in the hyperspace of a crowded Polish space $X$ is at most countable.
\end{proposition}

\begin{proof} Assume that some nwd-countable family $\C\subset\exp(X)$ is uncountable. We loss no generality assuming that $|\C|=\w_1$. Since $|\C|=\w_1<\add(\M)$, the union $\bigcup\C$ is meager and hence is contained in a meager $F_\sigma$-set $A=\bigcup_{n\in\w}A_n$ where each set $A_n$, $n\in\w$, is closed and nowhere dense in $X$. Since the set $\C=\bigcup_{n\in\w}\{C\in\C:C\cap A_n\ne\emptyset\}$ is uncountable, for some $n\in\w$ the set $\{C\in\C:C\cap A_n\ne\emptyset\}$ is uncountable too. This means that $\C$ is not nwd-countable.
\end{proof}

We do not know if Proposition~\ref{nwd-cc} can be generalized to weakly nwd-countable families.

\begin{problem} Is it consistent with ZFC that each weakly nwd-countable family $\C\subseteq\exp(\II)$ is countable?
\end{problem}

\section{Main Result}

In this section we combine the results proved in the preceding sections and obtain our main result:

\begin{theorem}\label{main} The product $\II\times\alpha D$ of the closed interval $[0,1]$ and the one-point compactification $\alpha D$ of a discrete space $D$ of cardinality $|D|\ge\non(\M)$ contains a closed subspace $X\subset \II$ which has weight $w(X)=\non(\M)$ and is not supercompact.
\end{theorem}

\begin{proof} 
Let $\mathcal{X}$ be a non-meager subset of $\exp(\II)$ of cardinality $\non(\M)$. Fix a subset
$E\subseteq D$ of cardinality $\non(\M)$ and enumerate the family $\mathcal X$ as $\mathcal X=\{X_i:i\in E\}$. Then $X=\II\times\{\infty\}\cup\bigcup_{i\in E}(X_i\times\{i\})$ is a closed subspace of $\II\times \alpha D$ of weight $w((X)=\non(\M)$, which is not supercompact according to Theorem~\ref{super}.
\end{proof}

\begin{theorem}\label{t5.2} Under $\cof(\M)=\w_1$ the product $\II\times\alpha D$ of the closed interval $[0,1]$ and the one-point compactification $\alpha D$ of a discrete space $D$ of cardinality $|D|\ge\omega_1$ contains a closed subspace $X\subseteq \II$ which has weight $w(X)=\w_1$, is first countable and hereditarily paracompact but not supercompact.
\end{theorem}

\begin{proof}
By the Theorem~\ref{disjoint}~(2) there is a non-meager nwd-countable subset $\mathcal{X}\subseteq \exp(\II)$ of cardinality $\omega_1$. This set can be enumerated as $\mathcal X=\{X_i:i\in E\}$ for some subset $E\subset D$ of cardinality $|E|=\w_1$. By Theorem~\ref{super}, the closed subspace $X=\II\times\{\infty\}\cup\bigcup_{i\in E}(X_i\times\{i\})$ of $\II\times\alpha D$ is not supercompact.
The nwd-countablity of $\mathcal{X}$ implies that $\mathcal{X}$ is both point-countable and weakly nwd-countable. Hence, by Proposition~\ref{pointc}, the space $X$ is first countable, and by Theorem~\ref{para}, the space $X$ is hereditarily paracompact.
\end{proof}

Combining Theorem~\ref{t5.2} with Proposition~\ref{p3.1}, we get:

\begin{corollary}\label{c5.3} Under $\cof(\M)=\w_1$ there is a compact space $X$
having the following properties:
\begin{enumerate}
\item $X$ has weight $w(X)=\w_1$;
\item $X$ is uniform Eberlein compact;
\item $X$ is Rosenthal compact;
\item $X$ is first countable;
\item $X$ is hereditarily paracompact;
\item $X$ is not supercompact.
\end{enumerate}
\end{corollary}

\begin{problem} Does there exist a ZFC-example of a compact spaces with properties (1)--(6) of  Corollary~\ref{c5.3}?
\end{problem}

It is known and easy to prove that for any supercompact spaces $X,Y$ their product $X\times Y$ again is supercompact.
Theorem~\ref{main} implies that in contrast to the supercompactness, the hereditary supercompactness is not productive.

\begin{corollary} The closed interval $\II$ and the one-point compactification $\alpha D=\{\infty\}\cup D$ of a discrete space of cardinality $|D|=\non(\M)$ are hereditarily supercompact but their product $[0,1]\times\alpha D$ is not.
\end{corollary}

\end{document}